\documentclass[11pt, a4paper, oneside]{article}
\usepackage[utf8]{inputenc}
\usepackage[T1]{fontenc}
\usepackage{amsmath, amssymb, amsfonts, amsthm, mathtools}
\usepackage{enumitem}
\usepackage[english]{babel}
\usepackage{mathrsfs}
\usepackage{natbib}
\usepackage{graphicx}
\usepackage{doi}
\usepackage{lmodern}
\usepackage{tablefootnote}
\usepackage{algorithm}
\usepackage{booktabs}
\usepackage{algpseudocode}
\usepackage{geometry}
\usepackage{hyperref}
\usepackage{float}
\usepackage{xcolor}
\usepackage{fancyhdr}
\usepackage{mathrsfs}
\usepackage{microtype}
\usepackage{setspace}
\usepackage{indentfirst}
\usepackage{mathtools}
\usepackage{xurl}
\definecolor{pucp}{rgb}{0, 0.125, 0.376}
\definecolor{codegreen}{rgb}{0,0.6,0}
\definecolor{codegray}{rgb}{0.5,0.5,0.5}
\definecolor{codepurple}{rgb}{0.58,0,0.82}
\definecolor{backcolour}{rgb}{0.95,0.95,0.92}
\usepackage[hyphenbreaks]{breakurl}
\hypersetup{
  colorlinks = true,
  citecolor = blue,
  urlcolor = black,
  linkcolor = blue
}

\theoremstyle{plain}
\newtheorem{theorem}{Theorem}[section]
\newtheorem{lemma}[theorem]{Lemma}

\newtheorem{corollary}[theorem]{Corollary}

\theoremstyle{definition}

\newtheorem{example}[theorem]{Example}

\newtheorem{assumption}{Assumption}

\newcommand{\floor}[1]{\left\lfloor #1 \right\rfloor}
\newcommand{\ceiling}[1]{\left\lceil #1 \right\rceil}
\newcommand{\mailto}[1]{\href{mailto:#1}{\texttt{#1}}}

\newcommand{\norm}[1]{\left\lVert#1\right\rVert}

\begin{document}
\title{Congestion and Penalization in Optimal Transport}
\author{
Marcelo Gallardo
\thanks{Department of Mathematics, Pontificia Universidad Católica del Perú (PUCP). Gallardo acknowledges insightful discussions with Professor Federico Echenique (UC Berkeley),
and former Minister of Health of Peru, Aníbal Velásquez. Dr. Velásquez provided key information with respect to the Peruvian health system.}\\
\mailto{marcelo.gallardo@pucp.edu.pe}
\and
Manuel Loaiza
\thanks{Autodesk, Inc.}\\
\mailto{manuel.loaiza@autodesk.com}
\and
Jorge Chávez
\thanks{Ch\'avez acknowledges support from the Pontificia Universidad Cat\'olica del Per\'u.}\\
\mailto{jrchavez@pucp.edu.pe}
}

\date{\today}
\maketitle

\begin{abstract}
We introduce a novel model based on the discrete optimal transport problem that incorporates congestion costs and replaces traditional constraints with weighted penalization terms. This approach better captures real-world scenarios characterized by demand-supply imbalances and heterogeneous congestion costs. We develop an analytical method for computing interior solutions, which proves particularly useful under specific conditions. Additionally, we propose an \(O((N+L)N^2 L^2)\) algorithm to compute the optimal interior solution.
For certain cases, we derive a closed-form solution and conduct a comparative statics analysis.
Finally, we present examples demonstrating how our model yields solutions distinct from classical approaches,
leading to more accurate outcomes in specific contexts,
such as Peru's health and education sectors.
\end{abstract}

\newpage

\section{Introduction}

\indent Optimal Transport (OT) \citep{10.1007/978-3-540-71050-9, 10.2307/j.ctt1q1xs9h}
is a mathematical technique that, in recent years,
has been integrated into economic theory,
particularly in the study of matching markets
\citep{chiappori2010hedonic, hal-03936221, 10.1093/imaiai/iaz015, carlier2020sista, arXiv:2402.13378}.
Unlike classical matching models \citep{10.1080/00029890.1962.11989827, 10.1086/260756, 10.2307/1913392, 10.1017/CCOL052139015X, 10.1257/000282803321947001, 10.1257/0002828054825469, 10.1257/aer.20130929}, OT optimizes over distributions, providing a more flexible and general framework. Starting from the classical model, in which matching costs are represented by a linear function, various extensions have incorporated a regularization term in the objective function to obtain solutions with desirable properties such as sparsity. Notable examples include entropic regularization \citep{10.1086/677191, 10.1093/imaiai/iaz015, arXiv:2003.00855, hal-03936221} and quadratic regularization \citep{lorenz_manns_meyer_2019, gonzalez2024sparsity, wiesel2024sparsity, nutz_2024}. Both classical OT and its regularized variants have been widely applied in analyzing matching markets, including marriage markets \citep{10.1086/677191}, migration dynamics \citep{carlier2020sista}, labor markets \citep{10.3982/QE928}, and school choice \citep{arXiv:2402.13378}.

This paper introduces a new model built upon the quadratic regularization framework, similar to \cite{nutz_2024}, but adopting the approach of \cite{10.1007/s10589-023-00476-1} while introducing heterogeneity in the quadratic term. Our model captures elements absent in classical formulations and better aligns with real-world scenarios. Specifically, by replacing equality constraints with weighted penalization terms, the solution accommodates supply and demand imbalances, a feature particularly relevant in developing countries when modeling matching in education and healthcare markets.

Countries with developing economies often experience significant inefficiencies in education and healthcare due to excess demand, insufficient supply, mismatching, and systemic congestion. These structural issues have contributed to high mortality rates and service deficiencies, as demonstrated during the COVID-19 pandemic. For instance, \cite{JHU2023} indicates that Peru recorded the highest per capita COVID-19 mortality rate globally, exceeding 6,400 deaths per million inhabitants. Similar inefficiencies have been observed across Latin America, where restricted healthcare access exacerbates disparities.

A key factor behind these inefficiencies is that individuals are not properly matched due to physical barriers, bureaucratic issues, and congestion \citep{anaya2024fragmentation, Velasquez2020}, compounded by excess demand. In countries such as Peru, India, and Brazil \citep{state_of_jakarta_traffic_congestion}, congestion is particularly severe. For instance, \cite{WorldBank2024} estimates that traffic congestion alone costs Peru 1.8\% of its GDP annually. Given these conditions, accounting for congestion and excess demand is crucial when modeling these dynamics.

The model presented in this paper provides a framework for congestion costs while also capturing excess of demand across different institutional contexts. As such, it reflects the realities of many developing countries, contrasting with developed nations such as France or Switzerland, where robust transportation infrastructure, efficient bureaucratic systems, and policies ensure universal access to education and healthcare.

The remainder of this paper is structured as follows. Section \ref{sec:preliminaries} defines the fundamental concepts and notation. Section \ref{sec:the-model} introduces the proposed model and examines its theoretical properties. Section \ref{sec:examples} presents illustrative examples that demonstrate the advantages of our approach. Due to data availability constraints, our empirical analysis focuses on the Peruvian health and education sectors. All proofs are provided in the Appendix.

\section{Preliminaries}
\label{sec:preliminaries}

We consider two sets, $X = \{x_1, \ldots, x_N\}$ and $Y = \{y_1, \ldots, y_L\}$. Each element $x_i$ $(y_j)$ represents an individual or a group of individuals/entities that share certain properties and are grouped into the same cluster. For example, in the marriage market (where usually $N = L$), $X$ is the set of men and $Y$ is the set of women. In the case of school matching, $X$ consists of groups of students, grouped, for instance, according to their district, and $Y$ is the set of schools. We denote by $\mu_i$ the \emph{mass} of $x_i$ and by $\nu_j$ the \emph{mass} of $y_j$. For instance, in the marriage market, $\mu_i = \nu_j = 1$, while in the case of schools, $\nu_j$ would represent the capacity of school $j$. Analogously, if $X$ were patients and $Y$ medical care centers, then parameters $\nu_j$ would represent the capacity of the medical care center. When referring to an element of $X$, instead of denoting it by $x_i$, we usually, to simplify the notation, refer to it by $i$. Analogously, the elements of $Y$ are referred to by the index $j$, instead of $y_j$. Moreover, we denote the set of indices $\{1, \ldots, N\}$ by $I$ and the set of indices $\{1, \ldots, L\}$ by $J$. Lastly, we denote by $\pi_{ij}$ the number of individuals of type $i$ matched with  $j$.

The problem addressed in the classic literature, from the perspective of a central planner, is to decide how many individuals from group $i$ should be matched with $j\in J$ and so forth for each $i$, minimizing the matching cost\footnote{Matching individuals incurs a cost that is not limited solely to <<physical>> transportation costs, which certainly accounts for both ways (round trip), but also encompasses implicit costs linked to specific characteristics of $i$ and $j$ such as tuition fee, entrance exam, languages, sex, age, etc. This is why we refer to them as matching costs instead of transportation costs.}, which is given by means of a function $C: \mathbb{R}_{+}^{N, L}\times \mathbb{R}^{P} \to \mathbb{R}$ depending on the matching $\pi=[\pi_{ij}] \in \mathbb{R}_{+}^{N, L}$\footnote{In this work, we will mostly assume that the number of individuals matched can take values in the real positive line and not only in the positive integers.  Note that this is the same issue that arises when one solves the utility maximization problem in the classical framework assuming divisible goods. Later on, we will address again this issue and explain why considering $\pi_{ij}\in \mathbb{R}_{+}$ allows drawing solid conclusions from an economic perspective.}, and a vector of parameters $\theta\in \mathbb{R}^P$. Moreover, the central planner must ensure that there are neither excesses of demand nor supply. Hence, the central planner solves
\begin{equation}\label{eq:optimization-C-general-case}
  \min_{\pi \in \Pi(\mu, \nu)} \ C(\pi; \theta),
\end{equation}
where
\begin{equation}\label{eq:set-Pi-mu-nu}
  \Pi(\mu, \nu) = \left\{\pi_{ij}\geq 0: \sum_{j=1}^{L}\pi_{ij} = \mu_i, \ \forall \ i \in I\ \wedge \
    \sum_{i=1}^{N}\pi_{ij} = \nu_j, \ \forall \ j\in J\right\}.
\end{equation}
A solution to \eqref{eq:optimization-C-general-case} will be from now referred to as an optimal matching or optimal (transport) plan, and will be denoted by $\pi^{*}$. In the standard optimal transport model, separable linear costs are assumed \citep{10.2307/j.ctt1q1xs9h}. This is, $C(\pi, \theta) = \sum_{i, j} c_{ij}\pi_{ij}$. It is therefore assumed that the marginal cost of matching one more individual from $i$ with $j$  is always the same, regardless of how many people are already matched and independent of any other variable.  Hence, the central planner seeks to solve
\begin{equation*}
\mathcal{P}_O: \ \min_{\pi \in \Pi(\mu, \nu)} \ \sum_{i=1}^N \sum_{j=1}^L c_{ij}\pi_{ij}.
\end{equation*}
To solve $\mathcal{P}_O$, one typically employs linear programming techniques, such as the simplex method. As discussed in the classical literature, the most general form of the OT problem allows for the existence of infinite types, and in such a case, the optimization is done over continuous distributions. In this paper, however, we are not going to study continuous distributions. What we do focus on, in line with the entropic regularization problem (see, for example, \cite{carlier2020sista} and \cite{peyre2019computational}), is working with a variation of the optimization problem in the discrete setting. In the case of entropic regularization \eqref{eq:entropic-regularization}, the problem addressed is
\begin{equation}\label{eq:entropic-regularization}
    \min_{\pi \in \Pi(\mu, \nu)} \ \sum_{i=1}^{N}\sum_{j=1}^{L} c_{ij}\pi_{ij} + \sigma \pi_{ij}\ln(\pi_{ij}),
\end{equation}
with $\sigma>0$. Given the strict convexity of $f(x)=x\ln x$, $f(0)=0$ and $\lim_{x \downarrow 0^{+}} f'(x) = -\infty$, the solution is interior, i.e. $\pi_{ij}^{*}>0$. Another variation is the quadratic regularization, where the problem becomes
\begin{equation}\label{eq:quadratic-regularization}
    \min_{\pi \in \Pi(\mu, \nu)} \ \sum_{i=1}^{N}\sum_{j=1}^{L} c_{ij}\pi_{ij} + \frac{\varepsilon}{2}||\pi||_2^2.
\end{equation}
Unlike the problem \eqref{eq:entropic-regularization}, in the case of \eqref{eq:quadratic-regularization}, interior solutions cannot be guaranteed\footnote{This is a common feature with our model, it is not straightforward to determine if the solution is interior.}. In the model we present in the following section, we build upon the problem \eqref{eq:quadratic-regularization}, making a considerable number of modifications that allow us to adapt to specific economic contexts of countries with structural problems. Before concluding this section, let us briefly note that, by a combinatorial argument, it is possible to conclude that the number of matchings is bounded by \( L^M \) in the case where $ \pi_{ij} \in \mathbb{Z}_{+}$. However, for the case $\pi_{ij} \in \mathbb{R}_{+}$, considering $\mu_i, \nu_j> 0$ for all $(i, j)\in I\times J$, the compactness of $ \Pi(\mu, \nu)$ and continuity of the objective functions, ensure the existence of a solution to $\mathcal{P}_O$ and its variants by Weierstrass Theorem.

\section{The model}
\label{sec:the-model}

In this section, we present the model that we propose, inspired by the optimal transport problem with quadratic regularization, but following the approach of \cite{10.1007/s10589-023-00476-1}. The model is derived from the very characteristics of the observed reality in certain locations. This will be explored in more detail in Section \ref{sec:examples}.

First, we need to allow the number of individuals of $X$ who belong to $i$ and are matched with $j=1, ..., L$, to not necessarily be $\mu_i$\footnote{We anticipate that $\mu_i$ will no longer be the mass of individuals of group $i$ but rather a targeted quota for individuals of group $i$.}. Similarly, it may be the case that not all those matched with $j$ sum up to $\nu_j$. This model allows for the possibility of excess supply or demand, which is reasonable in some contexts, as we will see. Indeed, underdeveloped countries may not be able to ensure full coverage in education and health, making it more realistic for them to face a trade-off. However, it is natural for the central planner to seek to minimize these excesses: ensuring that children attend school, that schools or hospitals do not become overcrowded, etc.

Mathematically, we model this by replacing the equality constraints defined by $\Pi(\mu, \nu)$ with penalties in the objective function. Moreover, we introduce weights for each penalty. That is, the constraint $\sum_{i=1}^{N}\pi_{ij} = \nu_j$ is replaced by the penalty term $\delta_j\left[\sum_{i=1}^{N}\pi_{ij} - \nu_j \right]^2$, with $\delta_j>0$, and the constraint $\sum_{j=1}^{L}\pi_{ij}=\mu_i$ is replaced by $\epsilon_i \left[\sum_{j=1}^{L}\pi_{ij}-\mu_i \right]^2$, with $\epsilon_i>0$. The parameters $\epsilon_i, \delta_j$ are weights. We could use any $p\geq 1$ norm for the penalization. However, the quadratic structure yields mathematical simplifications and fulfills the desired role. Then, by allowing deviations, as we will see in the examples, we better approximate the reality of developing countries that cannot fully ensure that demand perfectly matches supply.

Secondly, as is natural in some environments (see the next section), congestion costs are present.
These costs reflect the fact that matching more individuals from \(i\) with the same \(j\) becomes increasingly costly.
For example, from the perspective of physical transportation costs, in countries with high vehicular traffic congestion, the effect of increasing from \(x\) cars to \(x+1\) passing through a certain avenue is less or equal to increasing from \(x+n\) to \(x+n+1\) with \(n \geq 1\). Therefore, clustering groups based on geographic location means that matching many individuals from the same group to a single \( j\) congests the access route (which is the same). Hence, we introduce the term $\sum_{i, j}a_{ij}\pi_{ij}^2$ in the cost structure. The coefficient $a_{ij}$ captures heterogeneity\footnote{In some situations, the coefficient might be large, but in others—such as cases where there are few schools or hospitals, lightly congested streets, good traffic lights, etc.—the coefficient is small. Moreover, one could question whether adding a car still marginally increases costs when a route is already saturated. However, this effect would only arise when the number of travelers is excessively high relative to the route's capacity. For simplicity, we omit this case, as modeling a function that is initially quadratic and later constant would overly complicate the analysis when applying first order conditions.}, while the quadratic term represents the previously described phenomenon\footnote{Instead of using $\pi_{ij}^2$, we could consider a general strictly increasing and convex function $\psi$, such as $\psi(\pi_{ij}) = e^{\pi_{ij}}$ or $\pi_{ij}^3$. However, the quadratic structure facilitates quantitative analysis and preserves the consistency of the results and modeling.}. Note that quadratic costs are not limited to physical transportation costs but can also represent bureaucratic costs. A hospital receives patients of the same type. As more patients of this type arrive, the system must process a greater number of cases. Since they share the same characteristics, it is assumed that the same computer or system will handle their processing. Given the precarious conditions in developing countries, increasing from $x$ to $x+1$ patients may not significantly affect the system, but increasing from $x+n$ to $x+n+1$ with $n>1$ might (e.g., leading to system freezes, delays, etc.).

Finally, we certainly have $\pi_{ij} \geq 0$, for all $(i, j) \in I \times J$.
However, we do not impose upper bounds since we consider a population or universe that is arbitrarily large (a subpopulation of a sufficiently large country)\footnote{This considerably simplify our analysis and does not affect the logic of the model.}. Hence, the optimization is carried out over the entire space $\mathbb{R}_{+}^{NL}$.  This phenomenon also explains the penalties: we no longer assume a fixed number of individuals of type $i$,  and $\mu_i$ represents now a target that the central planner aims to achieve (how many individuals of type $i$ should ideally be matched). Similarly, the parameters $\nu_j$ are also targets of the central planner.

Therefore, following the described scenario, the central planner seeks to minimize costs while taking into account the objective of reaching the targets $\mu_i$ and $\nu_j$. The tradeoff is controlled through a parameter $\alpha \in [0, 1]$. According to what has been specified, the problem is:
\begin{equation}\label{eq:PIC-problem-no-eq-restriction-2}
 \mathcal{P}_{CP}: \  \min_{\pi_{ij}\geq 0} \underbrace{\left\{\underbrace{\alpha\sum_{i=1}^{N}\sum_{j=1}^{L} \varphi(\pi_{ij}; \theta_{ij})}_{\text{Matching direct cost.}} + \underbrace{(1-\alpha)\left[\sum_{i=1}^{N}\epsilon_i \left(\sum_{j=1}^{L}\pi_{ij}-\mu_i \right)^2 + \sum_{j=1}^{L}\delta_j \left(\sum_{i=1}^{N}\pi_{ij} - \nu_j\right)^2\right]}_{\text{Costs of social objectives.}}  \right\}}_{
 F(\pi; \theta, \alpha, \epsilon, \delta, \mu, \nu).}
\end{equation}
where
$\epsilon_1, ..., \epsilon_N$, $\delta_1, ..., \delta_L$  and $\mu_1, \cdots, \mu_N$, $\nu_1, \cdots, \nu_L$ are all non negative, and
\begin{equation}\label{eq:varphi-dij-cij-aij}
  \varphi(\pi_{ij}; \theta_{ij}) = d_{ij}+c_{ij}\pi_{ij} + a_{ij}\pi_{ij}^2.
\end{equation}
In Equation \ref{eq:varphi-dij-cij-aij}, despite its practical relevance, the term $d_{ij}$, representing fixed costs, does not influence the resolution of the problem. For this reason, when considering the parameter vector $\theta_{ij}\in \mathbb{R}^2$, we think of it as $(c_{ij}, a_{ij})$. Unlike more recent models in the quadratic regularization literature, we allow heterogeneity in the quadratic structure.

Having now established the model, which, to the best of our knowledge, is new in the literature\footnote{Quadratic regularization does not involve penalization terms and assumes $a_{ij}=\varepsilon$ for all $(i, j)\in I\times J$. With respect to the classical optimal transport problem, linear costs are considered. On the other hand, entropic regularization involves Inada's conditions, which do not appear in our model. Finally, in \cite{10.1007/s10589-023-00476-1}, only general results concerning penalization are given and this particular problem is not studied at all.}, we focus in this section on the following theoretical problems: (i) ensuring the existence of a solution, (ii) analyzing uniqueness, (iii) addressing why optimization in $\mathbb{R}_{+}^{NL}$ is reasonable and why we do not resort to integer optimization, (iv) studying how to compute interior solutions, and (v) analyzing particular cases both from the analytical and numerical perspective. In the next section, we compare our model with previous ones from the literature and highlight its advantages and the new insights it provides.

\vspace{0.1in}

\noindent \textbf{Existence and uniqueness:} Regarding the existence of a solution to $ \mathcal{P}_{CP}$, in order to apply Weierstrass theorem to overcome the potential issue that the optimization is carried over an unbounded set, we can actually restrict the optimization to $\mathbb{R}_{+}^{NL}\cap \Omega$, where
\begin{equation*}
    \Omega =  [0, R]^{NL}, \ \text{with} \     R = N\max_{1\leq i\leq N}\{\mu_i\}  + L\max_{1\leq j\leq L}\{\nu_j\}.
\end{equation*}
In fact, it is clear from the cost function $F$ that it is strictly lower in the interior of $\Omega$ or in the axes than when evaluated in $\partial \Omega$ (without considering the axes) or outside $\Omega$. This is a consequence of the coercivity of the objective function \citep{rockafellar1970convex}. With respect to uniqueness, it is a consequence of the strict convexity of the objective function. Indeed, the objective function is the sum of a strictly convex function, $\sum_{i, j}\varphi(\pi_{ij}, \theta_{ij})$, with $N+L$ convex functions of the form $\varrho\left(\sum_{m=1}^{M}\eta_{m} - \Theta \right)^2$, with $\varrho, \Theta, \eta_m \in \mathbb{R}_{+}$.

\vspace{0.1in}

\noindent \textbf{Optimization carried over $\mathbb{R}_{+}^{NL}$:} As we mentioned previously, similarly to the case of the classical demand theory, we are assuming that $\pi_{ij}\in \mathbb{R}_{+}$. However, just as it does not make sense to consume $\sqrt{2}$ cars, it can be also unreasonable to consider that $\pi_{ij}$ is not restricted to taking values in $\mathbb{Z}_{+}$, since it ultimately represents the number of individuals. However, given the structure of the optimization problem—a convex quadratic optimization problem—following the classical literature on rounding methods \citep{BeckFiala1981} and, in particular, the discrepancy between integer \citep{ParkBoyd2017, DelPia2021} and continuous solutions in the case of separable quadratic functions with linear constraints \citep{Hochbaum1990}, it is possible to establish bounds on the deviation of the optimal solution when transitioning from the continuous domain \(\mathbb{R}_{+}^{NL}\) to the integer lattice \(\mathbb{Z}_{+}^{NL}\), and ensure that it is sufficiently close. The bound depends on the eigenvalues of the Hessian matrix of the objective function\footnote{Specifically, the deviation is bounded by \(||\pi_{\text{int}} - \pi^*||_{\infty} \leq O(\vartheta(H))\), where \(\vartheta(H) = \lambda_{\max}(H)/\lambda_{\min}(H)\) is the condition number.}. Solving the problem in \(\mathbb{R}_{+}^{NL}\) allows the use of nonlinear convex optimization techniques, yielding not only computational advantages but also analytical insights. In this work, we do not delve deeply into this aspect, but we emphasize that by adjusting the parameters, it is possible to control the bound on the norm of the difference between the solutions in the lattice and the Euclidean space.

\vspace{0.1in}

\noindent \textbf{Interior solutions:} For the sake of simplicity,
we take $\alpha=1/2$.  Karush Kuhn Tucker (KKT) first order conditions applied to \eqref{eq:PIC-problem-no-eq-restriction-2} yield
\begin{equation}\label{eq:F-pikl}
  \frac{\partial F}{\partial \pi_{ij}} = \frac{1}{2}\left(\varphi'(\pi_{ij}^{*}; \theta_{ij}) + 2\epsilon_i \left(\sum_{\ell=1}^{L} \pi_{i\ell}^{*}- \mu_i\right) + 2\delta_{j} \left(\sum_{k=1}^{N}\pi_{kj}^{*} - \nu_{j} \right)  - \gamma_{ij}^{*}\right)=0, \ \forall \ (i, j)\in I\times J.
\end{equation}
Here, $\gamma_{ij}$ is the associated multiplier to the inequality constraint $\pi_{ij}\geq 0$.
Determining whether or not the solution is interior, is not trivial. For corner solutions,
we have to iterate all possible combinations of $\gamma_{ij}^{*}$ equal or not to zero.
Formally, \(2^{NL}\) possibilities.
In general, the problem can numerically be solved.
In what follows, unless the contrary is stated,
we will address the case where the solution is interior.
In this case, from KKT,
we know that $\gamma_{ij}^{*}=0$ for all $(i, j) \in I\times J$. Hence, from \eqref{eq:F-pikl}, we have $\nabla F(\pi^{*})=0$. This set of equations can be written in the compact form $A\begin{bmatrix}
      \pi_{11}^{*} &
      \pi_{12}^{*} &
      \cdots &
      \pi_{NL}^{*}
    \end{bmatrix}^T = b$, where
\begin{equation}\label{eq:A=D+E+F}
  A = \underbrace{\text{Diag}(a_{11}, a_{12}, \dots, a_{NL})}_{D}
  + \underbrace{\text{Diag}(\epsilon_1, \dots, \epsilon_N) \otimes \mathbf{1}_{L\times L}}_{E}
  + \underbrace{\mathbf{1}_{N\times N} \otimes \text{Diag}(\delta_1, \dots, \delta_L)}_{F},
\end{equation}
and $ b = \left[\epsilon_1\mu_1+\delta_1\nu_1-c_{11}/2, \epsilon_1\mu_1+\delta_2\nu_2-c_{12}/2, \cdots, \epsilon_N\mu_N+\delta_L\nu_L-c_{NL}/2\right]^T.$ The following lemma states that $A$ is an invertible matrix.

\begin{lemma}\label{lemma-det-A-positive-IR}
The determinant of $A$ is strictly positive, whenever all parameters are strictly positive.
\end{lemma}

Therefore, the linear system $A\pi=b$ has a unique solution.
What we still don't know is whether or not this solution belongs to $\mathbb{R}_{++}^{NL}$.
If so, given the strict convexity of $F$, we would have determined, through an ex-post analysis,
the unique solution to $\mathcal{P}_{CP}$.
However, it may not always be the case that $A^{-1}b \in \mathbb{R}_{++}^{NL}$,
and it is not a trivial matter to determine. Under specific cases, we will be able to do this. We propose both an analytical and a computational method to solve $A\pi = b$. The analytical method allows us, in special cases, to derive important theoretical conclusions, such as closed-form solutions, bounds, and perform comparative statics.
From a computational perspective, we compare our algorithm,
which exploits the structure of the matrix \(A\),
with others for solving linear systems.

\subsection{Neumann's series approach}

\begin{assumption}\label{assumption-espectral-norm}
Let \(a_{ij}> 0\) for all \((i, j)\in I\times J\). Assume that
\begin{align*}
  \max_{1 \leq i \leq N} \{\epsilon_i\} \cdot L +
  \max_{1 \leq j \leq L}\{\delta_j\} \cdot N <
  \min_{(i, j) \in I \times J} \left\{a_{ij}\right\}.
\end{align*}
\end{assumption}

Assumption \ref{assumption-espectral-norm} implies that convex transport costs are large.
Moreover, the fact that \(\epsilon_i, \delta_j\) are small follows from their interpretation as normalized weights, i.e., \(\epsilon_i, \delta_j \in [0, 1]\).

\begin{lemma}\label{th-pi-menor-1}
Under Assumption \ref{assumption-espectral-norm}, the following holds
\begin{equation*}
  A^{-1} = \left(\sum_{k=0}^{\infty}(-1)^k(D^{-1}X)^k\right)D^{-1}.
\end{equation*}
\end{lemma}

\begin{theorem}\label{th-convergence-pi-n}
Under Assumption \ref{assumption-espectral-norm}, $\lim_{n\to \infty}\pi_n =  \pi^{*}=A^{-1}b$, where
\begin{align*}
   \pi_n=S_nD^{-1}b=\left(\sum_{k=0}^n (-1)^k(D^{-1}X)^{k}\right) D^{-1} b.
\end{align*}
\end{theorem}

\subsection{Special cases}

For the aim to explicitly compute $A^{-1}$, we need to impose some additional mild assumptions.

\subsubsection{No interest in overcrowding or no quotas.}

\begin{assumption}\label{assumption-rho-F-0}
Assume that $\delta_j=0$ for all $j\in J$ and $D=\beta I$ for some $\beta>0$.
\end{assumption}

Assumption \ref{assumption-rho-F-0} illustrates the case where the central planner does not care
if in over or underfilling schools or hospitals (\(F=0\)),
and convex costs are the same across the pairs $(i, j)$: $a_{ij}=\beta$.
For instance, the latter applies when distances, routes, or bureaucratic systems are almost the same
for all \((i, j)\in I\times J\).

\begin{assumption}\label{L-assumption-rho}
Assume that \(L\epsilon_i < \min\{1, \beta\}\) for all \(1 \leq i \leq N\).
\end{assumption}

In line with Assumption \ref{assumption-espectral-norm},
Assumption \ref{L-assumption-rho} applies when convex transport costs are large.

\begin{theorem} \label{th-F=0}
Under Assumptions \ref{assumption-rho-F-0} and \ref{L-assumption-rho}, $A^{-1}$ is given as follows
\begin{equation}\label{eq:A-1-epsilon-beta}
A^{-1} = \frac{I}{\beta} + \frac{1}{\beta}
\text{Diag}\left(
-\frac{\epsilon_1}{\beta+L\epsilon_1},
\dots,
-\frac{\epsilon_N}{\beta+L\epsilon_N} \right)
\otimes\mathbf{1}_{L\times L}.
\end{equation}
\end{theorem}

A similar result can be obtained by setting \(E=0\), i.e., when the central planner is only concerned with overcrowding or underutilization of facilities and does not care about population quotas.

\begin{corollary} \label{cor-pij-F-0}
Under  Assumptions \ref{assumption-rho-F-0} and \ref{L-assumption-rho}, the solution of $\mathcal{P}_{CP}$ is given by
\begin{equation}\label{eq:pij-simplified-case}
    \pi_{ij}^{*} = \frac{b_{ij}}{\beta} - \sum_{\ell = 1}^{L} \frac{b_{i\ell}\epsilon_i }{\beta^2+L\epsilon_i\beta},
\end{equation}
provided that the right-hand side of \eqref{eq:pij-simplified-case} is positive.
\end{corollary}

\begin{proof}
This result follows directly from the computation of \(A^{-1}b\)
by using \eqref{eq:A-1-epsilon-beta}.
\end{proof}

\subsubsection{Equal weighting and identical convex costs.}

\begin{assumption}\label{as:Y}
Let \(\rho\) and \(\zeta\) be real numbers such that
\(\rho > 2NL\zeta > 0\),
with \(a_{ij} = \rho\)
and \(\epsilon_i = \delta_j = \zeta\) for all $(i, j) \in I\times J$.
\end{assumption}

Assumption \ref{as:Y} implies that  the central planner assigns equal weight to each social objective and where congestion and bureaucratic costs are the same for each pair. Under this assumption, we have
\(D = \rho I\) and \(X = \zeta Y\), where the entries of
\(Y\) are given by
\[
Y_{ij} = \begin{cases}
    2 & i = j, \\
    1 & i \neq j \land (\ceiling{i/N} = \ceiling{j/N} \lor i \equiv j \pmod{N}), \\
    0 & \text{otherwise}.
\end{cases}
\]
This allows us to write
\[
A^{-1} = \frac{1}{\rho}
\left(
\sum_{k=0}^{\infty} \left(-\frac{\zeta}{\rho}\right)^k Y^k\right).
\]
Under Assumption \ref{as:Y},
we will be able to establish bounds on the optimal matching,
i.e., to bound the number of individuals matched across the pairs \((i, j)\).
Lemmas \ref{lem:max-yk}, \ref{lem:min-yk} and \ref{lemma-bound-aij-pij}
are used to establish Theorem \ref{th-control-Aij}.

\begin{lemma}\label{lem:max-yk}
Let \(k \geq 1\) be a positive integer. Then
\[
\max_{1 \leq i, \ j \leq NL}
\left\{\left(Y^k\right)_{ij}\right\}
\leq
\frac{\left(2NL\right)^k}{NL}.
\]
\end{lemma}

\begin{lemma}\label{lem:min-yk}
Let \(k \geq 2\) be a positive integer.Then
\[
\frac{(NL)^{\floor{k/2}}}{NL} \leq
\min_{1 \leq i, \ j \leq NL}
\left\{\left(Y^k\right)_{ij}\right\}.
\]
\end{lemma}

\begin{lemma}\label{lemma-bound-aij-pij}
Under Assumptions \ref{assumption-espectral-norm} and \ref{as:Y},
the lower and the upper bounds  of
\(\left(A^{-1}\right)_{ij}\) can be expressed in terms of
\(N, L, \zeta\) and \(\rho\),
\begin{equation}\label{eq-C1-C2}
C_1(N, L, \zeta, \rho) \leq (A^{-1})_{ij} \leq C_2(N, L, \zeta, \rho),
\end{equation}
where
\begin{align*}
C_1 & =
\frac{\zeta  \left(4 \zeta  N^3 L^3 \left(2 \zeta ^3-2 \zeta  \rho ^2-\rho ^3\right)+8 N^2 L^2 \rho ^2 \left(\rho ^2-\zeta ^2\right)+\zeta NL\rho ^2 (2 \zeta +\rho )-2 \rho ^4\right)}{\rho ^4
\left(\zeta ^2 NL-\rho ^2\right)
\left(2NL-1\right)
\left(2NL+1\right)
}\\
C_2 & =
\frac{\zeta ^2 NL \rho  (4 N L - 1)}
{\left(\rho^2-\zeta ^2 NL\right)
\left(\rho-2NL\zeta\right)
\left(\rho+2NL\zeta\right)}.
\end{align*}
\end{lemma}


\begin{theorem} \label{th-control-Aij}
Under Assumptions \ref{assumption-espectral-norm} and \ref{as:Y}, it follows that \(\pi_{ij}^{*} \leq NL\tilde{C}\), for all \(\ (i, j) \in I\times J\), where
$$
\tilde{C} =
\max\{|C_1|, C_2\} \cdot
\max_{\substack{1\leq i\leq N\\ 1 \leq j \leq L}}
\left\{
\left|(\epsilon_i\mu_i+\delta_j\nu_j)-\frac{c_{ij}}{2}\right|
\right\}.
$$
\end{theorem}

Theorem \ref{th-control-Aij} is of particular interest as it allows us to determine, without computing the inverse of $A$, the maximum number of individuals that would be matched between two points $i, j$. In practice, this enables, for example, the establishment of capacity constraints on routes or spaces.

\subsection{Algorithm for computing \(\pi^*\)}

We now provide an efficient algorithm to compute \(\pi^{*}\in \mathbb{R}_{++}^{NL}\).
This is established in Theorem \ref{th:algorithm-ONL}.
First, let us rewrite matrix $A$ given in \eqref{eq:A=D+E+F} as follows:
\begin{align}\label{eq:A-diag-re-written}
A & = \text{Diag}(a_{11},\dots,a_{NL})
+ \sum_{i=1}^N \left(\epsilon_i^{1/2} \mathbf{e}_i \otimes \mathbf{1}_{L \times 1}\right)
\left(\epsilon_i^{1/2} \mathbf{e}_i^T \otimes \mathbf{1}_{1 \times L}\right) + \sum_{j=1}^L \left(\delta_j^{1/2} \mathbf{e}_j \otimes \mathbf{1}_{N \times 1}\right)
\left(\delta_j^{1/2} \mathbf{e}_j^T \otimes \mathbf{1}_{1 \times N}\right).\nonumber
\end{align}

\begin{theorem}\label{th:algorithm-ONL}
For interior solutions $\pi^{*}$,
Algorithm 1 computes $\pi^{*}$ in \(O((N+L)N^2L^2)\) time.
\end{theorem}

\begin{algorithm}[H]\label{al:smw}
\caption{\textsc{Optimize}
\((a, b,
\epsilon_1, \dots, \epsilon_N,
\delta_1, \dots, \delta_L
)\)
}
\begin{algorithmic}[1]
\State \textbf{Input:}
Matrices \(a \in \mathbb{R}_{++}^{NL}\),
\(b \in \mathbb{R}^{NL}\) and
parameters \(\epsilon_1, \dots, \epsilon_N, \delta_1, \dots, \delta_L
\in \mathbb{R}_{++}\)
\State \textbf{Output:}
\(\pi^* \in \mathbb{R}^{NL}\)
\State Initialize
\(A^{-1} \gets \text{Diag}(1/a_{11},\dots,1/a_{NL})
\in \mathbb{R}^{NL, NL}\)
\For{\(i \gets 1, \dots, N\)}
\State Define \(u^{(i)} \in \mathbb{R}^{NL}\)
by \(u^{(i)} \coloneqq
\epsilon_i^{1/2} \mathbf{e}_i \otimes \mathbf{1}_{L \times 1}\)
\State \(A^{-1} \gets
A^{-1} -
\dfrac{A^{-1} u^{(i)} u^{(i)T}A^{-1}}
{1 + u^{(i)T} A^{-1} u^{(i)}}
\)
via Sherman-Morrison formula
\EndFor
\For{\(j \gets 1, \dots, L\)}
\State Define \(v^{(j)} \in \mathbb{R}^{NL}\)
by \(v^{(j)} \coloneqq
\delta_j^{1/2} \mathbf{e}_j \otimes \mathbf{1}_{N \times 1}\)
\State \(A^{-1} \gets
A^{-1} -
\dfrac{A^{-1} v^{(j)} v^{(j)T}A^{-1}}
{1 + v^{(j)T} A^{-1} v^{(j)}}
\)
via Sherman-Morrison formula
\EndFor
\State \textbf{return} \(A^{-1} b\)
\end{algorithmic}
\end{algorithm}

\begin{table}[H]
    \begin{center}
        \begin{tabular}{|c|c|c|c|}
            \hline
            Time & Sparse \(A\) & Galactic & Authors \\
            \hline
            \(O(N^3L^3)\)& No & No & Gaussian Elimination \\
            \(O((NL)^{2.81})\) & No & No & \cite{Strassen1969} \\
            \(O((NL)^{2.331645})\) & Yes & Yes & \cite{PV24} \\
            \(O((NL)^{2.371339})\) & No & Yes & \cite{Alman2025} \\
            \(O((N+L)N^2L^2)\) & No & No & \textbf{This paper} \\
            \hline
        \end{tabular}
        \caption{Algorithms for solving our linear system.
        Assume \(A\) is sparse if it has \(\widetilde{O}(NL)\) nonzero entries.
        ``Galactic'' refers to an algorithm wonderful in its asymptotic behavior,
        but is never used to actual compute anything (\cite{Lipton2010}).}
        \label{table:algorithm-comparison}
    \end{center}
\end{table}

\vspace{-0.2in}

It was observed by \cite{VV2015} that
inversion can be reduced to multiplication with an equivalent runtime
for \cite{Strassen1969} and \cite{Alman2025}.
Even though \cite{PV24} and \cite{Alman2025} provide the best bounds,
they are impractical due to large constants,
leaving us with the remaining three algorithms for practical purposes.
Among these, when \(L = \Theta(N)\),
our algorithm has the tightest upper bound compared to classical Gaussian elimination and an inversion derived from Strassen multiplication.

\subsection{Comparative statics}

Although we know how to compute \(\pi^{*}\) through Neumann’s series or Algorithm 1, obtaining a closed-form expression for \(\pi_{ij}^{*}\) using these techniques is not straightforward. Therefore, to facilitate comparative statics, one possible approach is to approximate the matrix \(A^{-1}\) using Neumann’s series. First, assume that \(A^{-1} \simeq D^{-1}\). This simplification allows us to derive a closed-form expression for \(\pi_{ij}^{*}\), providing initial insights. Under the assumption \(A^{-1} \simeq D^{-1}\), we obtain:
\[
\pi_{ij}^{*} \simeq \frac{2(\epsilon_i \mu_i + \delta_j \nu_j)-c_{ij}}{2a_{ij}}.
\]
From this expression, it follows that \(\partial \pi_{ij}^{*}/\partial a_{ij}, \partial \pi_{ij}^{*}/\partial c_{ij} < 0\) and
\(\partial \pi_{ij}^{*}/\partial \epsilon_i,  \partial \pi_{ij}^{*}/\partial \delta_j\), \(\partial \pi_{ij}^{*}/\partial \mu_i, \partial \pi_{ij}^{*}/\partial \nu_j  > 0\).
These results align with standard economic intuition. However, under this rough approximation, we obtain \(\partial \pi_{ij}^{*}/\partial \theta_{k\ell} = 0\) for \((k, \ell) \neq (i, j)\), which is unrealistic since we expect a substitution effect. To improve upon this, consider a refined approximation:
\[
A^{-1} \sim D^{-1} - D^{-1} X D^{-1} = D^{-1} - (D^{-1})^2 X.
\]
From smooth comparative statics, if \(\pi^{*} \in \mathbb{R}_{++}^{NL}\) is an interior solution to \(\mathcal{P}_{CP}\) associated with the parameter vector \((\overline{\theta}, \epsilon, \delta, \mu, \nu) \in \mathbb{R}_{++}^{2NL} \times \mathbb{R}_{++}^{N} \times \mathbb{R}_{++}^{L} \times \mathbb{R}_{++}^{N} \times \mathbb{R}_{++}^{L}\), then:

\begin{equation}\label{eq:comparative-statics-UR-pij-thetakl}
  \left[\frac{\partial \pi_{ij}^{*}}{\partial \theta_{k\ell}} \right] = -A_{(\overline{\theta}, \epsilon, \delta, \mu, \nu)}^{-1} [I_{NL\times NL} \, | \, 2 \text{Diag}(\pi_{11}^{*}, \cdots, \pi_{NL}^{*})].
\end{equation}
Thus, under the approximation \( A^{-1} \sim D^{-1}-(D^{-1})^2X \), we obtain:
\begin{equation}\label{approx-1-D-X-smooth-CS}
  \left[\frac{\partial \pi_{ij}^{*}}{\partial \theta_{k\ell}}\right]= \left[\frac{\partial \pi_{ij}^{*}}{\partial c_{k\ell}} \, \bigg| \, \frac{\partial \pi_{ij}^{*}}{\partial a_{k\ell}}\right] \simeq -\left[D^{-1}-(D^{-1})^2X \, | \, A_{\Pi, 2}^{-1} \right],
\end{equation}
where \(A_{\Pi, 2}^{-1}\) consists of multiplying column \(ij\) of \(D^{-1} - (D^{-1})^2X\) by \(\pi_{ij}^{*}\). From \eqref{approx-1-D-X-smooth-CS}, if \(\max_{i, j}\{\epsilon_i+ \delta_j\} < 1\), then:
\(\partial \pi_{ij}^{*}/\partial \theta_{ij} < 0\) for all \((i, j) \in I \times J\),
 \(\partial \pi_{ij}^{*}/ \partial \theta_{k\ell} > 0\) for \(i \neq k\) and \(j = \ell\) or \(i = k\) and \(j \neq \ell\),
 \(\partial \pi_{ij}^{*}/ \partial \theta_{k\ell} = 0\) if \(i \neq k\) and \(j \neq \ell\).  Then, we conclude from \eqref{approx-1-D-X-smooth-CS} that:
\[
\partial \pi_{ij}^{*} / \partial c_{ij} = - (1 - (\epsilon_i+\delta_j)) / a_{ij}^2 < 0,
\]
\[
\partial \pi_{ij}^{*} / \partial c_{i\ell} = \epsilon_i / a_{ij}^2 > 0, \quad \partial \pi_{ij}^{*} / \partial c_{kj} = \delta_j / a_{ij}^2 > 0, \quad \partial \pi_{ij}^{*} / \partial c_{k\ell} = 0 \text{ if } i\neq k, j\neq \ell.
\]
\[
\partial \pi_{ij}^{*} / \partial a_{ij} = -2\pi_{ij}^{*}(1-(\epsilon_i+\delta_j)) / a_{ij}^2 < 0, \quad \partial \pi_{ij}^{*} / \partial a_{i\ell} = 2\pi_{i\ell}^{*} \epsilon_i / a_{ij}^2 > 0,
\]
\[
\partial \pi_{ij}^{*} / \partial a_{kj} = 2\pi_{kj}^{*} \delta_j / a_{ij}^2 > 0, \quad \partial \pi_{ij}^{*} / \partial a_{k\ell} = 0 \text{ if } i\neq k, j\neq \ell.
\]
These results are much closer to what we would expect. Indeed, we now observe a \emph{substitution effect}: if the cost of matching individuals of type \(i\) with \(j\) increases ceteris-paribus, then the number of individuals of type \(i\) matched with \(\ell\) (where \(\ell \neq j\)) increases. However, it is important to note that these results are obtained under a truncated Neumann series approximation, and should be interpreted accordingly—as an approximation. Nevertheless, note that under Assumptions \ref{assumption-espectral-norm}, \ref{assumption-rho-F-0}, and \ref{L-assumption-rho}, it is possible to compute the effects of the parameters directly using \eqref{eq:pij-simplified-case}. In such case, similar conclusions can be derived.

\subsection{Case $N=L$}

The case \( N = L >1\) is particularly important in the classical literature on the marriage market \citep{10.1017/CCOL052139015X}.
Similarly, as we will see in Section \ref{sec:examples},
it is of particular interest when analyzing the healthcare sector in Peru.
If the solution in our model is interior,
the problem reduces to solving a system of linear equations,
and the condition \(N=L\) improves the upper bound on the number of operations required by
Algorithm 1 compared to folklore linear system solvers.
On the other hand, classical transportation problems and their variants
require approximation algorithms for solving
convex optimization problems in finite dimensions (\cite{arXiv:2003.00855}).

\section{Examples and applications}
\label{sec:examples}

\subsection{Health care}
The Peruvian healthcare system is characterized by being a fragmented system with three main types of medical care centers: SIS (Seguro Integral de Salud), EsSalud, and EPS (Entidades Prestadoras de Salud) \citep{anaya2024fragmentation}. EPS corresponds to private health insurance offered by companies such as Rimac, Mapfre, Pacífico, La Positiva, among others. These insurances are aimed at formal workers seeking additional coverage beyond mandatory insurance. On the other hand, EsSalud is the public health insurance financed by contributions from formal workers and employers, both from the private and public sectors. Finally, SIS is a universal public insurance targeting people in poverty, informals, or without the ability to pay EPS. For the year of the pandemic (2020), SIS and EsSalud together covered more than 80\% of the population,
while less than \(10\%\) was covered by EPS, see Table \ref{table:afiliados}.

\begin{table}[H]
    \begin{center}
        \begin{tabular}{|c|c|}
            \hline
            \textbf{Insurance} & \textbf{Covered people} \\
            \hline
            EPS       & 8\% \\
            EsSalud   & 30\% \\
            SIS       & 53\% \\
            \hline
        \end{tabular}
        \caption{Percentage of enrollees in Peru's healthcare system by type of medical care center
        in 2020, before COVID-19. At that time, Peru's population was 32,838,579 \citep{DataCommons2025}.
        }
        \label{table:afiliados}
    \end{center}
\end{table}

Under normal circumstances, an individual insured by SIS cannot be simultaneously enrolled in EsSalud or an EPS, and vice versa. The only permitted association is between EsSalud and EPS, where private insurance acts as a complementary coverage to the public system \citep{anaya2024fragmentation, Velasquez2020}. Ideally, an optimal allocation would ensure that informal workers are covered by SIS, while formal workers are appropriately distributed between EsSalud and EPS. However, in practice, overlapping affiliations occur, and individuals often seek medical care outside their designated system. Furthermore, a similar issue arises when categorizing healthcare utilization by type of illness: specialized medical centers create unintended overlaps in patient distribution across insurance networks. Additional issues related to congestion and deficiencies are detailed in Table \ref{table:patient_allocation}.

\begingroup
\renewcommand{\arraystretch}{1.5}
\begin{table}[H]
    \begin{center}
        \begin{tabular}{ |p{5cm}|p{5cm}|p{5cm}|  }
            \hline
            \textbf{Identified Problem} & \textbf{Quantifiable Indicator} & \textbf{Source} \\
            \hline
            Shortage of medical personnel in primary healthcare. &
            12 doctors per 10,000 inhabitants,
            far from the WHO-recommended standard of 43. &
            \cite{bendezu2020}. \\
            Lack of hospital beds in Peru’s healthcare system. &
            1.6 beds per 1,000 inhabitants, below the regional average. &
            \cite{WorldBank2020}.\\
            Congestion in neonatal intensive care units in public hospitals &
            50\% of units experience inefficiency due to patient overcrowding. &
            \cite{arrieta2017congestion}. \\
            Inefficiencies in patient referral system. &
            High percentage of patients treated in facilities not equipped for their conditions.\tablefootnote{In 2016, the MINSA (Ministry of Health) reported a shortage of over 47,000 healthcare professionals. Additionally, 36\% of medium and high-complexity facilities lacked sufficient personnel, 44\% did not have adequate equipment, and 25\% had infrastructure deficiencies.} &
            \cite{Soto2019}. \\
            Coverage noncompliance, high waiting times, and
            some values of medical performance per hour out of range. &
            Coverage of up to 86\% for certain complex treatments. &
            \cite{EsSalud2025b}.\\
            Deferrals in certain cities are very high. &
            More than 23\% of appointments were postponed (Jan-Mar 2025). &
            \cite{EsSalud2025}.\\
            \hline
        \end{tabular}
    \end{center}
    \caption{Issues in patient allocation within Peru's healthcare system.}
    \label{table:patient_allocation}
\end{table}
\endgroup

Given Table \ref{table:patient_allocation}, it is evident that Peru's healthcare system faces significant issues, including service inefficiencies, congestion costs, and saturation. Our model effectively captures these elements, unlike traditional matching models. Our approach can help identify critical areas for improvement, optimizing healthcare demand coverage and reducing congestion costs by analyzing the effect of parameters over $\pi^{*}$. It allows for the prioritization of interventions to address the most severe inefficiencies. To achieve this, estimating parameters is essential. This aligns with empirical research such as \cite{EcheniqueDoval2024} and the methodologies outlined in \cite{echenique2023empirical}, which provide a structured framework to evaluate these inefficiencies.

In Example \ref{ex:PCP-aij},
we simulate three groups of patients in three healthcare networks (SIS, EsSalud, EPS).
Group 3 consists of individuals who can afford an EPS for high-complexity care.
High-complexity care refers to a set of less frequent and more complex health interventions,
such as advanced surgical procedures and oncological treatments.
Group 2 consists of formal workers who can only use EsSalud for high-complexity care.
Note that they are not excluded from affording an EPS,
but if they have one, it will be used exclusively for low-complexity care.
Group 1 consists of the remaining individuals, including informal workers.

A particular edge case in Group 1 includes
wealthy individuals engaged in illegal activities (e.g., drug traffickers or
businessman avoiding taxes).
These individuals are informal workers but may still afford an EPS.
The central planner reasonably operates under the assumption that such cases do not exist. Moreover, it operates assuming no overlaps.\footnote{It is important to emphasize that our model is designed to be executed at a specific point in time. Thus, the planner does not seek overlaps, and therefore, they are not enabled in the model.}

Groups 1 and 3 exhibit significant differences in characteristics, such as socioeconomic status, which increases the cost of mismatching between them. The cost is even higher when there are bureaucratic or legal frictions, as seen in the case of groups 1 and 2, where an EsSalud insured individual cannot be covered simultaneously by SIS, and vice versa \citep{anaya2024fragmentation}. Our model accounts for this heterogeneity in costs, recognizing that legal constraints impose significantly higher penalties than other sources of mismatching. For instance, while receiving treatment for a simple illness at a high-complexity facility incurs some inefficiency, the cost associated with legal barriers preventing access to appropriate healthcare is substantially greater. Moreover, incorporating penalties and weighted constraints allows the model to capture excess demand effectively. Unlike the solutions in traditional models (see Example \ref{ex:PQ-aij}), our model (Example \ref{ex:PCP-aij}) assigns almot zero or one to the match between groups 1 and 2.

Example \ref{ex-full-epsilon} highlights the flexibility of our model by introducing $\varepsilon_1, \dots, \varepsilon_N$ and $\delta_1, \dots, \delta_L$. In the Peruvian context, the government may prioritize patients from EsSalud due to its connection to formal employment, resulting in higher weights assigned to the constraint related to $\mu_2$. On the other hand, the goal is to prevent SIS from becoming overcrowded while maximizing facilities utilization. This objective is achieved, as the example shows that row 2 and column 1 bear the highest load without exceeding $\mu_i$ or $\delta_j$, with respect to the other rows and columns (proportionally to the target mass).

In Example \ref{ex:excess-demand}, we set
$\sum_{i=1}^{N} \mu_i > \sum_{j=1}^{L} \nu_j$, which is crucial for an appropriate representation of excess demand, but additionally. Quadratic costs exacerbate the excess demand.
The observed effect, due to the intentionally chosen parameters, reflects that almost no one from group 2 is matched. The parameters can certainly be adjusted to obtain more realistic values. The example illustrates how our model effectively captures excess demand, a present phenomenon in the Peruvian reality, see Table \ref{table:patient_allocation}.

\subsection{Education}

The education system in Peru is highly complex due to its high degree of decentralization at both the primary and higher education levels. While this decentralization aims to improve educational management, it has generated significant disparities between urban and rural regions \citep{alvarez2010}. Only a few subsystems, such as the High-Performance Schools (COAR), maintain a centralized management model, ensuring homogeneous standards \citep{alcazar2021evaluation}. However, despite not being a centralized system - which would make our model better suited - the level of congestion in Lima and its impact on education justify the introduction of a strictly convex structure. Moreover, since not everyone enrolls in school, partly due to geographic and access limitations, the penalties are well-founded, instead of restrictions.

Specifically, in Peru, infrastructure disparities and access constraints have affected educational equity \citep{alcazar2021evaluation}. Geographic barriers, particularly the Andes and the Amazon rainforest, exacerbate these inequalities by severely limiting accessibility. These mobility constraints directly impact school attendance, contributing to persistent enrollment gaps, especially in secondary education \citep{alba2025opportunity}. Tables \ref{tabla_matricula_primaria} and \ref{tabla_matricula_secundaria} illustrate the evolution of enrollment rates in primary and secondary education, showing gradual improvement but persistent urban-rural disparities.
\begin{table}[H]
    \centering
    \small
    \renewcommand{\arraystretch}{1.1}
    \setlength{\tabcolsep}{5pt}
    \begin{tabular}{|c|c|c|c|c|c|}
        \hline
        \textbf{Area} &
        \textbf{2021} &
        \textbf{2022} &
        \textbf{2023} &
        \textbf{2024} &
        \textbf{Variation 2024/2023} \\
        \hline
        National  & 87.1  & 91.3  & 91.3  & 96.0  & 4.7\% \\
        Urban    & 87.1  & 91.2  & 91.7  & 96.7  & 5\% \\
        Rural     & 87.1  & 91.7  & 89.8  & 93.6  & 3.8\% \\
        \hline
    \end{tabular}
    \caption{Net enrollment rate in primary education in Peru (2021-2024) \citep{inei2024}.}
    \label{tabla_matricula_primaria}
\end{table}
\begin{table}[H]
    \centering
    \small
    \renewcommand{\arraystretch}{1.1}
    \setlength{\tabcolsep}{5pt}
    \begin{tabular}{|c|c|c|c|c|c|}
        \hline
        \textbf{Area} & \textbf{2021} & \textbf{2022} & \textbf{2023} & \textbf{2024} & \textbf{Variation 2024/2023} \\
        \hline
        National  & 80.1  & 81.5  & 86.0  & 88.7  & 2.7\% \\
        Urban    & 80.7  & 81.4  & 86.7  & 88.2 & 1.5\% \\
        Rural     & 78.1  & 81.8  & 83.6  & 90.0  & 6.4\% \\
        \hline
    \end{tabular}
    \caption{Net enrollment rate in secondary education in Peru (2021-2024) \citep{inei2024}.}
    \label{tabla_matricula_secundaria}
\end{table}
\vspace{-0.15in}
A comprehensive study on the impact of congestion on enrollment is provided by \cite{alba2025opportunity}\footnote{Alba found that the 17\% reduction in travel time (equivalent to 30 minutes per day) increased the enrollment rate by 6.3\%.}, highlighting its significance, in line with the findings of \cite{agarwal2019revealed}, thus, justifying the relevance of our model. Indeed, congestion is a major issue in Peru’s education system, particularly in urban areas. According to \cite{WorldBank2024}, Lima is one of the most congested cities in Latin America. It suffers from severe traffic bottlenecks that disproportionately affect students from lower-income districts \citep{alba2025opportunity}. When large numbers of students travel from the same location to the same school, the primary roads connecting them become saturated, increasing commuting times.

Thus, the Peruvian education system is characterized by lack of access, excessive demand, and limited supply, combined with sensitivity to physical traffic congestion, in contrast to certain education systems, such as the French one \citep{Eurydice2024, EducationFrance2024}, which is centralized, homogeneous, ensures universal education, and benefits from a much more modern transportation system. Therefore, the model we propose is well-suited to represent this situation (other cities with congestion such as Mumbai, Jakarta or São Paulo \citep{state_of_jakarta_traffic_congestion} could also be studied). Traditional OT models, by imposing the condition $\sum_i\mu_i=\sum_j\nu_j$, do not apply as effectively. Our model is crucial because, in countries or cities with constraints, allowing for supply or demand imbalances—i.e., schools not reaching full capacity or not all students being enrolled—is a more realistic assumption.

Example \ref{ex-education-full} is key to understanding the education case. We consider four student groups (\(N=4\)) and three schools (\(L=3\)). The groups represent: wealthy high-achieving students (\(i=1\)), poor high-achieving students (\(i=2\)), wealthy low-achieving students (\(i=3\)), and poor low-achieving students (\(i=4\)). School \( j=1 \) is top-ranked and expensive, \( j=2 \) has an average ranking and a mid-range price, and \( j=3 \) is lower-ranked but more affordable. Transportation costs reflect the greater commuting difficulties faced by poor students, who usually use public transportation that runs along the most congested main avenues \citep{alba2025opportunity}, while linear costs capture preferences, ensuring that better students prefer better schools while weaker students do not, controlling also by monetary cost. The solutions highlight key differences: \(\mathcal{P}_{CP}\) introduces quadratic penalties, leading to assignments where students with fewer resources, for whom matching is more costly due to their location and the assigned mode of transportation (as transportation in their area is precarious), are not matched. In contrast, those who have better facilities (positive correlation between socioeconomic status and the quality of transportation) are matched more easily. Moreover, high-achieving wealthy students are never matched with low-cost, low-quality institutions, and low-achieving poor students are never matched with the top, expensive school. Hence, our model captures the complications arising from transportation costs and the unfortunate reality that education cannot be guaranteed for everyone. For example, Peru's geography excludes certain populations in the highlands and jungle, making it very costly for the central planner to complete the match. In Example  \ref{ex-education-full}, 70\% of the top wealthy students are matched, but only almost  3 out of 10 of the poorer, less top-performing students are matched. In this case, both the linear and quadratic models capture the fact that preferences result in \( 0 \) individuals from group \( i=1 \) being matched to \( j=3 \). However, once again, they do not provide the flexibility for \( \sum_j \pi_{ij}^{*} \neq \mu_i \), required in some contexts.

\section{Conclusions}
\label{sec:conclusions}

This paper introduces a novel framework for analyzing mismatching, congestion effects, and supply-demand imbalances in developing economies matching markets. Our model extends the classical optimal transport framework by incorporating heterogeneous quadratic regularization and penalty terms for deviations from target allocations. Unlike traditional approaches that impose strict equality constraints, our formulation allows for more realistic depictions of inefficiencies, capturing excess demand, underutilization, and the role of heterogenous congestion costs. We have also analyzed the resulting optimization problem in detail, establishing conditions for the existence and uniqueness of solutions. Furthermore, we propose both analytical and computational methods to effectively compute interior solutions. Our approach provides not only theoretical insights but also practical tools for addressing real-world mismatching and congestion issues.

In summary, our model provides considerable flexibility, allowing for heterogeneity in congestion costs, i.e., some \( a_{ij} \) could be very small. Removing restrictions enables a better approximation of the reality in developing countries, where equilibrium equations \( \Pi(\mu, \nu) \) do not hold uniformly.

Applying our model to Peru’s healthcare sector highlights its ability to explain observed inefficiencies, and provide more flexibility to the central planner when they cannot ensure matching the entire population adequately, which is common in developing or poor countries.
The fragmented nature of the public insurance system exacerbates mismatching, leading to suboptimal patient distribution and increased congestion in specific medical centers. Our framework captures these distortions by introducing quadratic congestion costs and penalizing deviations from optimal allocations. Although we have focused on the Peruvian case due to the aforementioned data availability constraints, the model can be applied to centralized matching situations with heterogeneous congestion costs and excess supply and demand.

Future research could extend this framework to dynamic settings, stochastic environments where parameters evolve over time (e.g., Markov Jump Linear Systems, since at different times of the day, traffic is less sensitive to new cars), and empirical validation using real-world matching data. Determining whether the solution is interior in terms of the parameters is not a trivial matter and remains to be explored. Furthermore, exploring policy implications, such as optimal subsidy structures or decentralized decision-making mechanisms, could provide valuable information to address inefficiencies in public service delivery.

Our model aims to provide central planners with a mathematically flexible tool to approximate allocation problems (without restricting solutions to the integer domain), while allowing for imbalances between supply and demand and incorporating congestion costs. This is particularly relevant in contexts where congestion costs are significant and where, unlike in highly developed countries, ensuring universal access to healthcare and education, as well as preventing the saturation of these services, remains a major challenge.

\newpage
\appendix

\section{Proofs}

\begin{proof}[Proof of Lemma \ref{lemma-det-A-positive-IR}]
First, $\text{det}(D) = \prod_{(i, j)\in I\times J}a_{ij}>0$, $\text{det}(E) = \text{det}(F) =0$.
On the other hand, the eigenvalues of $E$ are non-negative since
the eigenvalues of $\text{Diag}(\epsilon_1, ..., \epsilon_N)$ are $\epsilon_i>0$ and
the eigenvalues of $\mathbf{1}_{L\times L} $ belong to $\{0, L\}$.
Hence, the products of eigenvalues $\epsilon_i\cdot 0$ and $\epsilon_i \cdot L$ are non-negative,
and so, $E$ is positive semi-definite.
Similarly, $F$ is positive semi-definite.
Thus, $A$ is the sum of a diagonal and positive definite matrix and
two other symmetric and semi-positive definite matrices. According to \cite{zhan2005determinantal}\footnote{For Minkowski's determinant inequality and its generalizations, see \cite{marcus1970extension}, \cite{artstein2015asymptotic}.}
\vspace{-5mm} 
\[
\text{det}(A) = \text{det}(D+E+F) \geq \text{det}(D+E) + \text{det}(F)
\geq \text{det}(D) + \text{det}(E) + \text{det}(F) > 0.\qedhere
\]
\end{proof}

\begin{proof}[Proof of Lemma \ref{th-pi-menor-1}]
Let  \(A=D+X\),
where \(X=E+F\). Then,
\[
A^{-1} = (D+X)^{-1} = (I-(-1)D^{-1}X)^{-1}D^{-1}.
\]
Then, for all \( \lambda \in \sigma(D^{-1}X)\),
\(\lambda \leq \max_{i, j}\left\{1/a_{ij}\right\} \cdot (\lambda_{\max}^{E}+ \lambda_{\max}^{F})\), where $\lambda_{\max}^E  = \max_i \{\epsilon_i\}  \cdot L$ and $\lambda_{\max}^F = \max_j \{\delta_j\}  \cdot N.$
Thus, $\norm{D^{-1}X}_{\sigma}<1$
\footnote{$\norm{\cdot}_{\sigma}$ denotes the spectral norm.},
\begin{equation*}
 (I-(-1)D^{-1}X)^{-1}= \sum_{k=0}^{\infty}(-1)^k(D^{-1}X)^k.
\end{equation*}
Then, by multiplying the series on the right hand side by \(D^{-1}\), the claim follows.
\end{proof}

\begin{proof}[Proof of Theorem \ref{th-convergence-pi-n}]
Define
\begin{equation*}
    \mathcal{E}_n = A^{-1}-S_n = \left(\sum_{k=n+1}^{\infty}(-1)^k(D^{-1}X)^k\right)D^{-1}.
\end{equation*}
On one hand
\(
\norm{\pi_n-\pi^{*}}_{\infty}
= \norm{\mathcal{E}_n b}_{\infty}
\leq \norm{\mathcal{E}_n b}_2
\). On the other hand,
\[
\norm{\mathcal{E}_n b}_2
\leq
\sqrt{NL}
\norm{\sum_{k=n+1}^{\infty}(-1)^k(D^{-1}X)^k}_{\sigma}
\norm{D^{-1}b}_{\infty}
\leq
\frac{\sqrt{NL} \norm{D^{-1}X}_{\sigma}^{n+1} \norm{D^{-1}b}_{\infty}}
{1-\norm{D^{-1}X}_{\sigma}}.
\]
Given \(\varepsilon > 0\), let
\[
N_{\varepsilon} =
\max
\left\{
1,
\ceiling{\left|\log_{\norm{D^{-1}X}_{\sigma}}
\left(
\frac{\varepsilon \left(1 - \norm{D^{-1}X}_{\sigma}\right)}
{\sqrt{NL} \norm{D^{-1} b}_{\infty}}
\right)\right|}
\right\}.
\]
For \(n \geq N_{\varepsilon}\), we have \(\norm{\pi_n-\pi^{*}}_{\infty} < \epsilon\).
\end{proof}

\begin{proof}[Proof of Theorem \ref{th-F=0}]
By using classical properties of Kronecker product, we have
\begin{align*}
    A^{-1}
    &=\frac{I}{\beta} +
    \left[\sum_{k=1}^{\infty} (-1)^k \left(\frac{1}{\beta}\right)^k
    (\text{Diag}(\epsilon_1, \dots, \epsilon_N) \otimes \mathbf{1}_{L\times L})^k\right]D^{-1}\\
    &=\frac{I}{\beta} + \frac{1}{\beta L}
    \sum_{k=1}^{\infty} (-1)^k \left(\frac{L}{\beta} \right)^k
    (\text{Diag}(\epsilon_1^k, \dots, \epsilon_N^k) \otimes  \mathbf{1}_{L\times L})\\
        &=\frac{I}{\beta} + \frac{1}{\beta L}
    \text{Diag}\left(\sum_{k=1}^{\infty}(-1)^k \left(\frac{L\epsilon_1}{\beta}\right)^k,
    \dots,
    \sum_{k=1}^{\infty}(-1)^k \left(\frac{L\epsilon_N}{\beta}\right)^k\right) \otimes \mathbf{1}_{L\times L}\\
    &=\frac{I}{\beta} + \frac{1}{\beta }
    \text{Diag}\left(-\frac{\epsilon_1}{\beta+L\epsilon_1},
    \dots, -\frac{\epsilon_N}{\beta+L\epsilon_N} \right)\otimes \mathbf{1}_{L\times L}.\qedhere
\end{align*}
\end{proof}

\begin{proof}[Proof of Lemma \ref{lem:max-yk}]
The claim certainly holds for \(k=1\).
Now, assuming it holds for \(k \geq 1\), it follows by induction that
\[
\max_{1 \leq i, j \leq NL}
\left\{\left(Y^{k+1}\right)_{ij}\right\}
=
\max_{1 \leq i, j \leq NL}
\left\{
\sum_{\ell=1}^{NL}
\left(Y^k\right)_{i \ell}
Y_{\ell j}
\right\}
\leq \sum_{\ell=1}^{NL} \frac{(2NL)^k}{NL} \cdot 2
= \frac{(2NL)^{k+1}}{NL}. \qedhere
\]
\end{proof}

\begin{proof}[Proof Lemma \ref{lem:min-yk}]
We have two distinct possibilities.

\noindent
\textbf{Case} \(k = 2m\) with \(m \geq 1\) \textbf{.}
We now proceed by induction.
We will manually verify that each
\(
\left(Y^2\right)_{ij}
= \displaystyle\sum_{\ell=1}^{NL} Y_{i \ell} \cdot Y_{\ell j}
\)
satisfies the inequality.
On the diagonal we have
\[
\left(Y^2\right)_{ii} =
\sum_{\substack{\ell = 1 \\ \ell \neq i}}^{NL} Y_{i \ell} \cdot Y_{\ell i}
+
Y_{ii}
\cdot
Y_{ii}
\geq 4.
\]
For \(i \neq j\), set
\[
\ell_0 = N \left(\ceiling{\frac{j}{N}} - \floor{\frac{i - 1}{N}} - 1\right) + i.
\]
Then \(\ell_0 \equiv i \pmod{N}\) and so \(Y_{i \ell_0} \geq 1\).
On the other hand,
\[
\ell_0
\in
\left[
N \left(\ceiling{\frac{j}{N}} - 1\right) + 1,
N \ceiling{\frac{j}{N}}
\right]
\]
implies \(\ceiling{\ell_0 / N} = \ceiling{j / N}\).
So, \(Y_{\ell_0j} \geq 1\).
It follows that
\[
\left(Y^2\right)_{ij}
= \sum_{\substack{\ell=1 \\ \ell \neq \ell_0}}^{NL} Y_{i \ell} \cdot Y_{\ell j}
+ Y_{i \ell_0} \cdot Y_{\ell_0 j} \geq 1.
\]
Assuming
\(
\min_{1 \leq i, j \leq NL}
\left\{\left(Y^{2m}\right)_{ij}\right\}
\geq
(NL)^m / NL
\)
holds for \(m \geq 1\), we obtain
\[
\min_{1 \leq i, j \leq NL}
\left\{
\left(Y^{2m + 2}\right)_{ij}
\right\}
=
\min_{1 \leq i, j \leq NL}
\left\{
\sum_{\ell=1}^{NL}
\left(Y^{2m}\right)_{i \ell}
\cdot
\left(Y^2\right)_{\ell j}
\right\}
\geq
\sum_{\ell=1}^{NL} \frac{(NL)^m}{NL}
=
\frac{(NL)^{m+1}}{NL}.
\]

\noindent
\textbf{Case} \(k = 2m + 1\) with \(m \geq 1\) \textbf{.}
We prove this by induction on \(m\) starting with the base case \(Y^3\):
\[
\left(Y^3\right)_{ij}
=
\sum_{\ell=1}^{NL}\left(Y^2\right)_{i \ell} \cdot Y_{\ell j}
=
\sum_{\substack{\ell=1 \\ \ell \neq j}}^{NL}
\left(Y^2\right)_{i \ell} \cdot Y_{\ell j}
+
\left(Y^2\right)_{i j}
\cdot
\, Y_{jj}
\geq 2.
\]
Assume the statement holds for \(m \geq 1\), then
\[
\min_{1 \leq i, j \leq NL}
\left\{
\left(Y^{2m + 3}\right)_{ij}
\right\}
=
\min_{1 \leq i, j \leq NL}
\left\{
\sum_{\ell=1}^{NL}
\left(Y^{2m + 1}\right)_{i \ell}
\cdot
\left(Y^2\right)_{\ell j}
\right\}
\geq
\sum_{\ell=1}^{NL} \frac{(NL)^m}{NL}
=
\frac{(NL)^{m+1}}{NL}.
\]
This completes the proof.
\end{proof}

\begin{proof}[Proof of Lemma \ref{lemma-bound-aij-pij}]
We write \(A^{-1}\) in terms of \(Y\)
\[
A^{-1}
= \frac{1}{\rho}
\left(
I
- \left(\frac{\zeta}{\rho}\right)Y
+ \sum_{m \geq 1} \left(\frac{\zeta}{\rho}\right)^{2m}Y^{2m}
- \sum_{m \geq 1} \left(\frac{\zeta}{\rho}\right)^{2m + 1}Y^{2m + 1}
\right)
\]
and apply Lemmas \ref{lem:max-yk} and \ref{lem:min-yk} to bound the series as follows,
\[
\frac{\zeta^2 N L}{\rho^2 - \zeta^2 NL}
\leq \sum_{m \geq 1} \left(\frac{\zeta}{\rho}\right)^{2m}\left(Y^{2m}\right)_{ij}
\leq \frac{4 \zeta^2 N^2 L^2}{\rho^2 - 4 \zeta^2 N^2 L^2}
\]
\[
\frac{\rho^3}{\rho(\rho^2 - \zeta^2 NL)} \leq
\sum_{m \geq 1} \left(\frac{\zeta}{\rho}\right)^{2m+1}\left(Y^{2m+1}\right)_{ij}
\leq \frac{8 \zeta^3 N^2 L^2}{\rho(\rho^2 - 4 \rho^2 N^2 L^2)}.
\]
Therefore, $(A_{ij})^{-1}$ is bounded from above by
\begin{equation*}
\frac{1}{\rho}
\left(
1
+ \frac{4 \zeta^2 N^2 L^2}{\rho^2 - 4 \zeta^2 N^2 L^2}
- \frac{\rho^3}{\rho(\rho^2 - \zeta^2 NL)}
\right),
\end{equation*}
and from below by
\begin{equation*}
\frac{1}{\rho}
\left(
- 2\left(\frac{\zeta}{\rho}\right)
+ \frac{\zeta^2 N L}{\rho^2 - \zeta^2 NL}
- \frac{8 \zeta^3 N^2 L^2}{\rho(\rho^2 - 4 \rho^2 N^2 L^2)}
\right).
\end{equation*}
From here, \eqref{eq-C1-C2} follows.
\end{proof}

\begin{proof}[Proof of Theorem \ref{th-control-Aij}]
By triangle inequality,
\begin{align*}
\pi^{*}_{ij}& \leq ||\pi^*||_{\infty}\\
&= \max_{\substack{1\leq i\leq N\\ 1 \leq j \leq L}}
\left\{
\left|
\sum_{k=1}^{NL}
\left(A^{-1}\right)_{(i-1)L+j \quad k} \cdot
b_{\ceiling{k/L} \quad k-L\floor{(k-1)/L}}
\right|
\right\}\\
&\leq
\sum_{k=1}^{NL}
\max_{\substack{1\leq i\leq N\\ 1 \leq j \leq L}}
\left|
\left(A^{-1}\right)_{ij}
\right|
\cdot
\max_{\substack{1\leq i\leq N\\ 1 \leq j \leq L}}
\left|
b_{ij}
\right|\\
&=
NL \tilde{C}.\qedhere
\end{align*}
\end{proof}

\begin{proof}[Proof of Theorem \ref{th:algorithm-ONL}]
Consider Algorithm 1.
It is easy to see that each prefix sum of \(A\) is invertible.
Hence, we can iteratively apply the Sherman-Morrison formula with a rank-1 update at each step.
Then, it is clear that Lines 3 and 12 take \(O\left(N^2 L^2\right)\).
First, the number of iterations for the for-loops on Lines 4-7 and 8-11 is \(N+L\).
We then show that each time we enter any for-loop,
the time spent is \(O\left(N^2 L^2\right)\).
Computing \(1+w^T A^{-1} w\) takes \(O\left(N^2 L^2\right)\),
so the only possible optimization is finding the optimal parenthesization for the product
\(A^{-1} w w^T A^{-1}\).
Since there are only five possible ways to parenthesize the expression,
we determine by brute force that computing
\((A^{-1}w)(w^T A^{-1})\) also takes \(O\left(N^2 L^2\right)\).
This implies the desired time complexity of \(O\left((N+L) N^2 L^2\right)\).
\end{proof}

\section{Numerical examples}

We define $\mathcal{P}_Q$ as the following optimization problem:
\begin{equation*}
\mathcal{P}_Q: \  \min_{\pi \in \Pi(\mu, \nu)} \ \sum_{i=1}^N \sum_{j=1}^L \varphi(\pi_{ij}, \theta_{ij}).
\end{equation*}
It is a generalization of the quadratic regularization problem.

\begin{example}\label{ex:PCP-aij}
The parameters used for solving \(\mathcal{P}_{CP}\) with \(d=5I_{3 \times 3}\) and \(\alpha = 0.5\) are
\[
c =
\begin{bmatrix}
1 & 50 & 20 \\
50 & 1 & 20 \\
20 & 10 & 1
\end{bmatrix},
\
a =
\begin{bmatrix}
1 & 5 & 10 \\
5 & 1 & 2 \\
10 & 5 & 1
\end{bmatrix},
\
\epsilon =
\delta =
\begin{bmatrix}0.3 \\ 0.3 \\ 0.3\end{bmatrix},
\
\mu =
\begin{bmatrix}100 \\ 50 \\ 20 \end{bmatrix}
\ \text{and} \
\nu =
\begin{bmatrix}90 \\ 40 \\ 40 \end{bmatrix}.
\]
The optimal solution \(\pi^*\)
obtained using Algorithm 1 in Mathematica 14.1
\footnote{We also ran \texttt{QuadraticOptimization} and verified that the optimal plans coincide.}
is
\[
\pi^{*} =
\begin{bmatrix}
34.7802 & 0.19412 & 1.65935 \\
0.10148 & 15.6978 & 3.41038 \\
0.883807 & 0.905689 & 9.65139
\end{bmatrix}.
\]
\end{example}

\begin{example}\label{ex:PQ-aij}
Using the same parameters as in \(\mathcal{P}_{CP}\) but enforcing the marginal constraints \(\Pi(\mu, \nu)\)
and removing penalization,
the optimal solutions to $\mathcal{P}_Q$ and $\mathcal{P}_O$ are
\[
\pi^{*}_{\mathcal{P}_Q} = \begin{bmatrix}
84.275 & 8.84062 & 6.88442 \\
4.2985 & 30.4206 & 15.2809 \\
1.42655 & 0.73873 & 17.8347
\end{bmatrix},
\ \pi^{*}_{\mathcal{P}_O} =\begin{bmatrix}
90 & 0 & 10 \\
0 & 40 & 10 \\
0 & 0 & 20 \\
\end{bmatrix}.
\]
\end{example}

\begin{example}\label{ex-full-epsilon}
Using the same parameters as in \(\mathcal{P}_{CP}\) but changing weighting to
$\epsilon = \begin{bmatrix}0.4 & 1 & 0.2\end{bmatrix}^T$
and $\delta = \begin{bmatrix}1 & 0.5 & 0.4\end{bmatrix}^T$ leads to
\[
\pi^{*}=  \begin{bmatrix}
50.7142 & 0.360177 & 1.75142 \\
4.56352 & 22.9044 & 7.05884 \\
2.37786 & 0.873057 & 9.57857
\end{bmatrix}.
\]
\end{example}

\begin{example}\label{ex:excess-demand}
Modifying the parameters with respect to Example \ref{ex-full-epsilon} as follows
\begin{equation*}
a = \begin{bmatrix}
1 & 20 & 2 \\
20 & 5 & 2 \\
5 & 2 & 0.5 \\
\end{bmatrix}, \
\mu = \begin{bmatrix}200 \\ 50 \\ 10\end{bmatrix} \
\text{and} \ \nu=\begin{bmatrix}100 \\ 20 \\ 50 \end{bmatrix}
\end{equation*}
yields
\[
  \pi^{*} = \begin{bmatrix}
69.4335 & 1.23953 & 19.2527 \\
1.52132 & 6.95671 & 11.9992 \\
3.14146 & 0.282174 & 7.55862
\end{bmatrix}.
\]
\end{example}

\begin{example}\label{ex-education-full}
Consider the following parameters for \(\mathcal{P}_{CP}\)
with \(d=\textbf{1}_{4 \times 3}\) and \(\alpha = 0.5\):
\[
c =
\begin{bmatrix}
0.1 & 1 & 6 \\
0.2 & 1 & 4 \\
4 & 1 & 0.2 \\
8 & 1 & 0.1
\end{bmatrix},
\
a =
\begin{bmatrix}
0.5 & 0.5 & 0.5 \\
2 & 2 & 1 \\
0.5 & 0.5 & 0.5 \\
2 & 2 & 1
\end{bmatrix},
\
\epsilon =
\begin{bmatrix}
0.2 \\
0.2 \\
0.2 \\
0.2
\end{bmatrix},
\
\delta =
\begin{bmatrix}
0.2 \\
0.2 \\
0.2
\end{bmatrix},
\
\mu =
\begin{bmatrix}
10 \\
10 \\
10 \\
10
\end{bmatrix},
\
\nu =
\begin{bmatrix}
10 \\
20 \\
10
\end{bmatrix}.
\]
The solution to the optimization problems are\footnote{In this example,
$\pi^{*}_{\mathcal{P}_{CP}}$ is not an interior solution.
Therefore, it is not possible to use Algorithm 1 to solve the problem.
Instead, we use \texttt{QuadraticOptimization}.}
\[
\pi^{*}_{\mathcal{P}_{CP}} =
\begin{bmatrix}
3.25505 & 3.89254 & 0 \\
1.20974 & 1.39412 & 0.333926 \\
0 & 3.99723 & 2.88862 \\
0 & 1.33717 & 2.17004
\end{bmatrix},
\
\pi^{*}_{\mathcal{P}_Q} =
\begin{bmatrix}
4.18 & 5.82 & 0 \\
3.25571 & 3.69071 & 3.05357 \\
1.25857 & 6.79857 & 1.94286 \\
1.30571 & 3.69071 & 5.00357
\end{bmatrix}
\]
and
\[
\pi^{*}_{\mathcal{P}_O} =
\begin{bmatrix}
10 & 0 & 0 \\
0 & 10 & 0 \\
0 & 10 & 0 \\
0 & 0 & 10
\end{bmatrix}.
\]
\end{example}

\newpage

\bibliographystyle{apalike}
\bibliography{references} 
\end{document}